\documentclass[10pt]{amsart}

\usepackage{amssymb,amsmath,amsthm,amsfonts}

\newcommand{\comment}[1]{}

\newtheorem{lem}{Lemma}

\newtheorem{cor}{Corollary}
\newtheorem{thm}[lem]{Theorem}
\newtheorem*{thmA}{Theorem A}
\newtheorem*{thmB}{Theorem B}
\newtheorem*{thmC}{Theorem C}
\newtheorem*{thmD}{Theorem D}

\theoremstyle{remark}
\newtheorem{rem}{Remark}

\theoremstyle{definition}

\newcommand{\Z}{\mathbb Z}

\newcommand{\N}{\mathbb N}

\newcommand{\E}{\epsilon}
\newcommand{\VE}{\varepsilon}

\newcommand{\T}{\theta}

\newcommand{\be}{\begin{equation}}
\newcommand{\ee}{\end{equation}}
\newcommand{\bee}{\begin{equation*}}
\newcommand{\eee}{\end{equation*}}

\begin{document}
\title{Arithmetic Structure in Sparse Difference Sets}
\author[Hamel \quad Lyall \quad Thompson \quad Walters]{Mariah Hamel \quad Neil Lyall \quad Katherine Thompson \quad Nathan Walters}

\begin{abstract}
Using a slight modification of an argument of Croot, Ruzsa and Schoen we establish a quantitative result on the existence of a dilated copy of any given configuration of integer points in sparse difference sets. More precisely, given any configuration $\{v_1,\dots,v_\ell\}$ of vectors in $\mathbb{Z}^d$, we show that if $A\subset[1,N]^d$ with $|A|/N^d\geq C N^{-1/\ell}$, then there necessarily exists $r\ne0$ such that $\{rv_1,\dots,rv_\ell\}\subseteq A-A$. 

\end{abstract}
\thanks{
This work was carried out while all four authors were members of the \emph{VIGRE Research Group in Arithmetic Combinatorics at the University of Georgia}. The first, third and fourth authors were partially supported by this NSF VIGRE grant. The second author was partially supported by NSF grant 0707099.}

\address{Department of Mathematics, The University of Georgia, 
Athens, GA 30602, USA}
\email{mhamel@math.uga.edu,\,lyall@math.uga.edu,\,thompson@math.uga.edu,\,nwalters@math.uga.edu}
\maketitle

\section{Introduction}

Many familiar theorems in mathematics have as a common feature the phenomenon that the set of differences from a sufficiently large set contains non-trivial structure. In this paper we study instances of this phenomenon in the finite setting of subsets of $\{1,2,\dots,N\}^d$. In this setting we are able obtain to quantitative structure theorems by adapting an 
argument of Croot, Ruzsa and Schoen \cite{CRS}.


\subsection{Arithmetic progressions in sumsets}

A good measure of the amount of additive structure in a given finite set of integers is provided by the size of the longest arithmetic progression that the set contains.  

In \cite{CRS} Croot, Ruzsa and Schoen establish, using only simple combinatorial arguments, the following structural result along these very lines for the sumset 
\be
A+B:=\{a+b\,:\,a\in A,\, b\in B\}
\ee
of two given sparse sets $A,B\subseteq[1,N]=\{1,2,\dots,N\}$.

\begin{thmA}[Croot, Ruzsa and Schoen \cite{CRS}]\label{CRS}
Let $A,B\subseteq[1,N]$ and 
$m\geq3$ be an odd positive integer.

If $|A||B|/N^2\geq 6 N^{-2/(m-1)}$, then $A+B$ contains an arithmetic progression of length at least $m$. 
\end{thmA}

We remark that if $|A||B|/N^2\geq (\log N)^{-1+\VE}$, for any $\VE>0$, then the conclusion of Theorem A can in fact be strengthened significantly. Using Fourier analytic techniques Green \cite{Green}, improving on previous work of Bourgain \cite{Bourgain}, proved the following

\begin{thmB}[Green \cite{Green}]\label{Green}
Let $A,B\subseteq[1,N]$, then there exist absolute constants $0<c<1<C$ such that if
\[\frac{|A||B|}{N^2}\geq C \frac{(\log\log N)^2}{\log N}\]
then $A+B$ contains an arithmetic progression of length at least $\exp\Bigl(c\left(\frac{|A||B|}{N^2}\log N\right)^{1/2}\Bigr)$. 
\end{thmB}

\comment{
This result vastly improves on bounds of the form $c\log\log\log\log\log N$ that can be obtained by a straightforward application of Gowers' quantitative bounds for Szemer\'edi's theorem\footnote{ \ Szemer\'edi's theorem: Subsets of positive upper density in the integers contain arbitrarily long arithmetic progressions.} \cite{Gowers}, but yet are obtained significantly more easily using arguments that rely on only standard (linear) Fourier analysis. }

Ruzsa \cite{Ruzsa2} has shown that for any $\VE>0$ and $N$ sufficiently large, there exists a set $A\subseteq[1,N]$ with $|A|/N\geq1/2-\VE$
whose sumset $A+A$ does not contain an arithmetic progression of length $\exp\bigl(\left(\log N\right)^{2/3+\VE}\bigr)$.

In this paper we give a slight simplification of the proof of Theorem A that appears in \cite{CRS} and show how this can then be easily adapted to prove the natural higher dimensional generalization of Theorem A (as well as certain polynomial variants). 

\comment{In this paper, we generalize the results of Croot, Ruzsa and Schoen in two ways.  In one direction, we consider $A\subset[1,N]^d$.  In another direction, we consider another measure of structure, namely the existence of polynomial configurations in difference sets.  Finally, we are able to incorporate both results into one statement in Theorem \ref{MDPoly}.  We state our result for higher dimensions here, and save the statements of the other results until we have been able to motivate them.}

\comment{
We note that this result is stronger than the theorem of Green (and Bourgain) when $|A||B|/N^2\ll (\log N)^{-1}$, as in this range Green's result does not give a non-trivial bound on the length of the longest arithmetic progression in $A+B$; however, when $|A||B|/N^2\geq (\log N)^{-1+\VE}$, Green's results give a significantly stronger conclusion.}

\comment{
\noindent
{\bf Open Question} (Croot){\bf .} \emph{Fix $m\geq1$. Given $N$, determine the largest $\T\in(0,1)$ such that if $A,B\subseteq[1,N]$ with $|A||B|/N^2\geq N^{-\T}$, then $A+B$ contains an arithmetic progression of length at least $m$.}

It follows from the results in \cite{CRS} that for all $N$ sufficiently large this largest $\T=\T(N)$ satisfies 
\[2/m\ll \T\ll_\VE(\log m)^{-3/2+\VE}.\]
}

\subsection{Multi-dimensional configurations in sparse difference sets}
The main result of this paper is the following quantitative multi-dimensional Szemer\'edi theorem for sparse difference sets. 

\begin{thm}\label{MDS}
Given any collection $v_1,\dots,v_\ell$ of vectors in $\Z^d$ there exists a constant $C_{\ell,d}=C_{\ell,d}(v_1,\dots,v_\ell)$ such that the \emph{difference set} of any set $A\subseteq[1,N]^d$ with 
\[|A|/N^d\geq C_{\ell,d} N^{-1/\ell}\] 
will be guaranteed to contain some dilate of the original configuration $\{v_1,\dots,v_\ell\}$, that is to say that there will necessarily exist an integer $r>0$ for which
\[\{rv_1,\dots,rv_\ell\}\subseteq A-A.\] 
\end{thm}

\begin{rem} Since $A-A$ is symmetric it will also contain the reflection in the origin of this configuration, namely the configuration $\{-rv_1,\dots,-rv_\ell\}$.
\end{rem}

A simple averaging argument (that we will make use of again and again) allows us to deduce, from Theorem \ref{MDS}, 
the following structural result for the sumset $A+B$ of two given sets $A,B\subseteq[1,N]^d$.

\begin{cor}\label{corMDS}
Let $v_1,\dots,v_\ell\in\Z^d$ and $A,B\subseteq[1,N]^d$.
If $|A||B|/N^{2d}\geq 2^dC_{\ell,d} N^{-1/\ell}$, then there exists $r\ne0$ and $t\in[2,2N]^d$ such that \[t\pm r\cdot\{v_1,\dots,v_\ell\}\subseteq A+B.\]
\end{cor}

\begin{proof}
Since
\[\sum_{t\in[2,2N]^d}\left|B\cap(t-A)\right|=|A||B|\]
it follows that there exists $t\in[2,2N]^d$ such that if we set $D=B\cap(t-A)$, then
\[|D|\geq \frac{|A||B|}{(2N-1)^d}.\]
The result therefore follows immediately from Theorem \ref{MDS} since
$D-D+t\subseteq  A+B$.
\end{proof}

\begin{rem}[On the constant $C_{\ell,d}$ in Theorem \ref{MDS} and Corollary \ref{corMDS}]
Given any collection $v_1,\dots,v_\ell$ of vectors in $\Z^d$ let $s$ denote the size of the largest projection of these given vectors onto the coordinate axes. In other words we define
\be\label{s1}
s=s(v_1,\dots,v_\ell)=\max_{1\leq j\leq\ell}\|v_j\|_\infty=
\max\left\{\,\left|\langle v_j,e_i\rangle\right|\,:\,1\leq j\leq\ell,\, 1\leq i\leq d\right\},
\ee
where $\{e_i\}_{1\leq i\leq d}$ denotes the standard basis vectors for $\Z^d$.

It is clear that in order to obtain a non-trivial conclusion in Theorem \ref{MDS} and Corollary \ref{corMDS} we must have $r\in[1,N/s]$, and in particular $N\geq s$ (a fact that will be forced on us by the choice of constant $C_{\ell,d}$). 

In the proof of Theorem \ref{MDS}, which we present in Section \ref{proofs}, we will see that we can in fact take
\be
C_{\ell,d}:=\Bigl(2s\prod_{i=1}^d\prod_{j=1}^\ell\Bigl(1+\frac{\left|\langle v_j,e_i\rangle\right|}{s}\Bigr)\Bigr)^{1/\ell}\leq 2^d(2s)^{1/\ell}.
\ee

It follows that we are able to recover Theorem A, the structural result for sparse sumsets of Croot, Ruzsa and Schoen, with only marginally weaker bounds (the constant 6 replaced with 8) as a special case of Corollary \ref{corMDS}, namely the special case $d=1$ with $v_j=j$. 

Finally, it is also easy to see that for specific (and interesting) choices of vectors $v_1,\dots,v_\ell\in\Z^d$ the true value of $C_{\ell,d}$ is significantly smaller than the trivial upper bound of $2^d(2s)^{1/\ell}$. For example, in the case when $\ell=d$ and $v_j=e_j$ 
(a $d$-dimensional corner), we have $C_{\ell,d}\leq 2^{1+1/\ell}$.
\end{rem}

\subsection{Outline of the paper}
In Section \ref{proofCroot} we present a proof of Theorem \ref{MDS} in the special case when $d=1$ with $v_j=j$. As a corollary of this result (and some known quantitative results on polynomial patterns in difference sets) we give quantitative bounds on the size of a set $A\subseteq[1,N]$ that will ensure that its difference set contains a long arithmetic progression whose common difference is a perfect square. 

In Section \ref{polyvar} we formulate a polynomial generalization of Theorem \ref{MDS}, namely Theorem \ref{MDPoly}, 
and present a proof of this result in a special case (Theorem \ref{SinglePoly}). 

Finally, in Section \ref{proofs} we prove Theorem \ref{MDS} and sketch the proof of Theorem \ref{MDPoly}.

\section{Arithmetic progressions in sparse difference sets}\label{proofCroot}


\subsection{A special case of Theorem \ref{MDS}}
Although the following result can be extracted from Croot, Ruzsa and Schoen \cite{CRS} we feel that the argument below is perhaps slightly simpler and easier to generalize.

\begin{thm}[A special case of Theorem \ref{MDS}]\label{Croot}
Let $A\subseteq[1,N]$ and 
$m\geq3$ be an odd positive integer.

If $|A|/N\geq 4 N^{-2/(m-1)}$, then $A-A$ contains an arithmetic progression of length at least $m$. 
\end{thm}

\begin{proof}[Proof of Theorem \ref{Croot}]
Let $m=2\ell+1$. For each
$w=(w_1,\dots,w_\ell)\in\Z^\ell$ we define\footnote{ \ Of course $N/\ell$ need not be an integer, however here, and in the remainder of this article, we will make the slight (but convenient) abuse of notation of identifying $[1,x]$ with $[1,\lfloor x \rfloor]$ for any given positive real number $x$.}
\[\mathcal{R}_w=\{r\in[1,N/\ell]\,:\,jr+w_j\in A \ (1\leq j\leq\ell)\}\]
and note that if, for some $w\in\Z^\ell$, there exist $r',r''\in\mathcal{R}_w$ with $r'\ne r''$, then it follows immediately that
\[
j(r'-r'')\in A-A\]
for each $1\leq j\leq\ell$ and hence, utilizing the fact that $A-A$ is symmetric, that the difference set $A-A$ contains an arithmetic progression of length $2\ell+1$.

It therefore suffices to establish the existence of a $w\in\Z^\ell$ such that $|\mathcal{R}_w|\geq2$.
In order to do this we (naturally) restrict our attention to those $w$ for which $\mathcal{R}_w$ has at least a chance of being non-empty, namely 
\[\mathcal{W}=\{w\in\Z^\ell\,:\,1-jN/\ell\leq w_j\leq N-1 \ (1\leq j\leq\ell)\},\]
and note that 
\[|\mathcal{W}|\leq N^\ell\prod_{j=1}^\ell(1+j/\ell)\leq 2^\ell N^\ell.\] 
Since the average
\[\frac{1}{|\mathcal{W}|}\sum_{w\in\mathcal{W}}|\mathcal{R}_w|=\frac{1}{|\mathcal{W}|}\sum_{w\in\mathcal{W}}\sum_{r=1}^{N/\ell}\prod_{j=1}^\ell 1_{A}(jr+w_j)=\frac{1}{|\mathcal{W}|}|A|^\ell\frac{N}{\ell}\]
it follows that there must exist a $w\in\mathcal{W}$ such that
\[|\mathcal{R}_w|\geq\left(\frac{|A|}{N}\right)^\ell\frac{N}{\ell2^\ell}\]
and consequently, for this choice of $w$, that the set $\mathcal{R}_w$ will contain at least 2 elements provided 
\[\frac{|A|}{N}\geq C_\ell\frac{1}{N^{1/\ell}}\]
where $C_\ell=2(2\ell)^{1/\ell}$.
It is an easy (calculus) exercise to finally show that $2\leq C_\ell\leq 4$.
\end{proof}

It is clear that the arguments presented above are flexible enough to be applied almost verbatim to more general situations and at this point the reader is encouraged to prove Theorem \ref{MDS} for herself (the details can be found in Section \ref{proofs}). 

We now turn our attentions to the problem of finding polynomial configurations in difference sets. 


\subsection{Some remarks on polynomial configurations in difference sets}

The following result, a quantitative polynomial Szemer\'edi theorem for difference sets, can be established by a careful application of standard Fourier analytic (circle method) techniques, see \cite{LM} and \cite{LM2}.

\begin{thmC}[Lyall and Magyar \cite{LM2}]\label{LM}
Let $P_1,\dots,P_\ell\in\Z[r]$ with each $P_j(0)=0$ and $k=\max_j \deg P_j\geq2.$

There exists an absolute constant $C=C(P_1,\dots,P_\ell )$ such that for any $A\subseteq[1,N]$ with
\[\frac{|A|}{N}\geq C\left(\frac{(\log\log N)^2}{\log N}\right)^{1/\ell(k-1)}\]
there necessarily exist $r\ne0$ such that
\[\{P_1(r),\dots ,P_\ell (r)\}\subseteq A-A.\]
\end{thmC}

In the case of a single polynomial ($\ell=1$), this result was also obtained by Lucier \cite{Lucier} and, to the best of our knowledge, constitutes the best bounds that are currently known for arbitrary polynomials with integer coefficients and zero constant term. However, using some rather involved Fourier arguments, Pintz, Steiger and Szemer\'edi in \cite{PSS} were able to establish the following impressive result for square differences.

\begin{thmD}[Pintz, Steiger and Szemer\'edi \cite{PSS}]\label{PSS}
Let $A\subseteq[1,N]$ and $m\geq 3$ be a positive integer. If
\[\frac{|A|}{N}\geq C\left(\frac{1}{\log N}\right)^{c\log\log\log\log N}\]
then there necessarily exist $r\ne0$ such that $r^2\in A-A.$
\end{thmD}

We note that it is conjectured that for any $\E>0$ and $N$ sufficiently large there exists a set $A\subseteq[1,N]$ with 
$|A|\geq N^{1-\E}$ that contains no square differences. Ruzsa \cite{Ruzsa} has demonstrated this for $\E=0.267$. 

Arguing as in the proof of Theorem \ref{Croot} we can deduce, from Theorem D, the following result pertaining to arithmetic progressions with square differences.

\begin{cor}\label{CorPSS}
Let $A\subseteq[1,N]$ and $m\geq 3$ be a positive integer. If
\[\frac{|A|}{N}\geq C'\left(\frac{1}{\log N}\right)^{2c'\log\log\log\log N/(m-1)}\]
then $A-A$ contains an arithmetic progression of length at least $m$ with common difference $r^2$ (with $r\in\N$).
\end{cor}

\begin{proof}
Let $m=2\ell+1$. 
For each
$w=(w_1,\dots,w_\ell)\in\Z^\ell$ we again define
\[\mathcal{R}_w=\{r\in[1,N/\ell]\,:\,jr+w_j\in A \ (1\leq j\leq\ell)\}.\]

It follows from Theorem D that if, for some $w\in\Z^\ell$ we have
\begin{equation}\label{eq1}
\frac{|R_w|}{N/\ell}\geq C\left(\frac{1}{\log N/\ell}\right)^{c\log\log\log\log (N/\ell)}
\end{equation}
then there will necessarily exist $r',r''\in\mathcal{R}_w$ and $r\ne0$ such that $r'-r''=r^2$ and consequently also that
\[jr^2\in A-A\]
for each $1\leq j\leq\ell$. Utilizing the fact that $A-A$ is symmetric, it then follows that the difference set $A-A$ contains an arithmetic progression of length $2\ell+1$ with a square common difference.

It therefore suffices to establish a condition on the set $A$ guaranteeing that estimate (\ref{eq1}) will hold.
We saw, in the proof of Theorem \ref{Croot} above, that there exists
$w=(w_1,\dots,w_\ell)\in\Z^\ell$ such that
\[\frac{|R_w|}{N/\ell}\geq \left(\frac{|A|}{N}\right)^\ell\frac{1}{2^\ell}.\] 
Hence, inequality (\ref{eq1}) will hold provided
\[\frac{|A|}{N}\geq 2C^{1/\ell}\left(\frac{1}{\log N/\ell}\right)^{c\log\log\log\log (N/\ell)/\ell}\]
and the result follows.
\end{proof}

The methods of Pintz, Steiger and Szemer\'edi were later extended by Balog, Pelik\'an, Pintz and Szemer\'edi \cite{BPPS} to obtain Theorem D (with the same bounds, but with the constant $C$ now depending on $k$), and hence also Corollary \ref{CorPSS}, for $k$th power differences. However, as we mentioned above, these impressive bounds have yet to be established for arbitrary polynomials with integer coefficients and zero constant term.

\section{A polynomial generalization of Theorem \ref{MDS}}\label{polyvar}

\subsection{A special case of Theorems \ref{SimpleMDPoly} and \ref{MDPoly}}
Before stating our polynomial generalization of Theorem \ref{MDS}, namely Theorem \ref{MDPoly} (which contains as a special case a weak variant of Theorem C for sparse difference sets), we present the following very special case. It is our hope that working through this special case will help motivate the ultimate formulation of Theorem \ref{MDPoly}.

\begin{thm}[A special case of Theorems \ref{SimpleMDPoly} and  \ref{MDPoly}]\label{SinglePoly} 
Let $P(r)=a_k r^k+\cdots+a_1 r+a_0$ with each $a_j\in\Z$ and $a_k>0$. 
If $N$ is sufficiently large and $A\subseteq[1,N]$ with  \[|A|/N\geq (4a_k^{1/k}) N^{-1/k}\] then there exist $r',r''\in\N$ with $r'\ne r''$ such that \[P(r')-P(r'')\in A-A.\]
\end{thm}

Note that a conclusion of the form $P(r)\in A-A$ for such sparse sets is forbidden, see the remark proceeding Theorem D concerning partial progress towards a conjecture of Ruzsa.

\begin{proof} 
It suffices to establish the existence of a $w\in\Z$ such that $|\mathcal{R}_w|\geq2$ where
\[\mathcal{R}_w=\{r\in[1,N_0]\,:\,P(r)+w\in A\}\]
with $N_0=(N/a_k)^{1/k}$.
We define $\mathcal{W}$ to be the \emph{smallest} collection of $w\in\Z$ for which 
\[P(r)+\mathcal{W}\supseteq[1,N]\]
for all $r\in[1,N_0]$.
If $N$ is sufficiently large (depending on the coefficients of $P$) it follows that 
\[\max_{r\in[1,N_0]}|P(r)|\leq P(N_0)\leq 2N\] 
and hence that
\[\mathcal{W}\subseteq\{w\in\Z\,:\,1-P(N_0)\leq w\leq P(N_0)-1\}\]
from which we can conclude that
\[|\mathcal{W}|\leq 4N.\]

\begin{rem}\label{sufflarge} If $P(r)>0$ on $[1,N_0]$, then we could conclude, as in the proof of Theorem \ref{Croot}, that in fact $|\mathcal{W}|\leq 2N$ (provided that $N$ is large enough).
In the further special case when $P(r)=a_k r^k$ we note that one can drop the ``let $N$ be sufficiently large'' assumption in the statement of the theorem.
\end{rem}

Since the average
\[\frac{1}{|\mathcal{W}|}\sum_{w\in\mathcal{W}}|\mathcal{R}_w|=\frac{1}{|\mathcal{W}|}\sum_{w\in\mathcal{W}}\sum_{r=1}^{N_0} 1_{A}(P(r)+w)\geq\frac{|A|N_0}{4N}\]
for $N$ sufficiently large, it follows that if
\[|A|/N\geq (4a_k^{1/k}) N^{-1/k}\]
and $N$ is sufficiently large, then there will necessarily exist $w\in\mathcal{W}$ for which $|\mathcal{R}_w|\geq2$.
\end{proof}

\subsection{A polynomial variant of Theorem \ref{MDS}}

Let $\{v_1,\dots,v_\ell\}$ be a fixed configuration of vectors in $\mathbb{Z}^d$. Given polynomials $P_1,\dots,P_\ell\in\Z[r]$ with $\max_{1\leq j\leq \ell}\deg P_j\leq k$, our methods allow us to establish the following \emph{non-isotropic} generalization of Theorem \ref{MDS}. 

\begin{thm}\label{SimpleMDPoly}
If $A\subseteq[1,N]^d$ with $|A|/N^d\geq CN^{-1/\ell k}$, for some constant $C$, then there exist $r',r''\in\N$ with $r'\ne r''$ such that \[\left\{\left(P_{1}(r')-P_{1}(r'')\right)v_1,\dots,\left(P_{\ell}(r')-P_{\ell}(r'')\right)v_\ell\right\}\subseteq A-A.\]
\end{thm}

Since, for each $1\leq j \leq \ell$, we can write
\[\left(P_{j}(r')-P_{j}(r'')\right)v_j=\sum_{i=1}^d \left(Q_{ij}(r')-Q_{ij}(r'')\right)e_i\]
where
$Q_{ij}(r):=P_j(r)\langle v_j,e_i\rangle$
we see that Theorem \ref{SimpleMDPoly} is a special case of the following more general result.

\begin{thm}\label{MDPoly} Let $\mathcal{Q}=\{Q_{ij}\}_{\substack{1\leq i\leq d \\ 1\leq j\leq \ell}}$ be a fixed collection of $\ell d$ polynomials in $\Z[r]$ and $k=\max_{Q_{ij}\in \mathcal{Q}}\deg Q_{ij}$, then
there exists a constant $C_{\ell,d,k}=C_{\ell,d,k}(\mathcal{Q})$ such that if $N$ is sufficiently large and $A\subseteq[1,N]^d$ with
\[|A|/N^d\geq C_{\ell,d,k} N^{-1/\ell k}\] then there exists $r',r''\in\N$ with $r'\ne r''$ such that \[\sum_{i=1}^d\left(Q_{ij}(r')-Q_{ij}(r'')\right)e_i\,\in A-A.\] 
\end{thm}

\begin{rem}[on the constant $C_{\ell,d,k}$ in Theorem \ref{MDPoly}]
If we now let $t$ denote the absolute value of the largest leading coefficient of the polynomials that have the largest degree, namely
\be\label{s2}
t=\max_{Q_{ij}\in \mathcal{Q}}\lim_{r\rightarrow\infty} |Q_{ij}(r)|/r^k
\ee
then,
as we shall see in the proof below, we will be able take
\[C_{\ell,d,k}=4^d(2t^{1/k})^{1/\ell}\]
in general and replace the 4 with 2 in the special case where the $Q_{ij}$ take only positive values on $\N$. 

Note that we \emph{exactly} recover Theorem \ref{MDS} (with the same constant) as a special case of Theorem \ref{MDPoly}, namely the special case where $Q_{ij}(r)=r\langle v_j,e_i\rangle$, since in this case we can drop the ``let $N$ be sufficiently large'' assumption (see Remark \ref{sufflarge} in the proof of Theorem \ref{SinglePoly}), we leave this observation for the reader to verify. 
 
\end{rem}

\section{Proof of Theorem \ref{MDS} and Theorem \ref{MDPoly}}\label{proofs}

\subsection{Proof of Theorem \ref{MDS}}
Let $s$ be defined by formula (\ref{s1}). For each
$w=(w_1,\dots,w_\ell)\in(\Z^d)^\ell$ we define
\[\mathcal{R}_w=\{r\in[1,N/s]\,:\,rv_j+w_j\in A \ (1\leq j\leq\ell)\}\]
and $\mathcal{W}$ to be the \emph{smallest} collection of $w=(w_1,\dots,w_\ell)\in(\Z^d)^\ell$ for which 
\[\{rv_j+w_j\,:\, w\in\mathcal{W}\}\supseteq[1,N]^d\]
for all $1\leq j\leq\ell$ and $r\in[1,N/s]$.

As in the proof of Theorem \ref{Croot} it will suffice to establish the existence of a $w\in\mathcal{W}$ such that $|\mathcal{R}_w|\geq2$.
The fact that this happens whenever
\[\left(\frac{|A|}{N^d}\right)^\ell\geq\frac{2s}{N}\prod_{i=1}^d\prod_{j=1}^\ell\Bigl(1+\frac{\left|\langle v_j,e_i\rangle\right|}{s}\Bigr)\]
follows immediately from the observation that the average
\[\frac{1}{|\mathcal{W}|}\sum_{w\in\mathcal{W}}|\mathcal{R}_w|=\frac{1}{|\mathcal{W}|}\sum_{w\in\mathcal{W}}\sum_{r=1}^{N/s}\prod_{j=1}^\ell 1_{A}(rv_j+w_j)=\frac{1}{|\mathcal{W}|}|A|^\ell\frac{N}{s}\]
and the (easily verified) fact that
\[|\mathcal{W}|\leq N^{d\ell}\prod_{i=1}^d\prod_{j=1}^\ell\Bigl(1+\frac{\left|\langle v_j,e_i\rangle\right|}{s}\Bigr).\qedhere\] 

\subsection{Proof of Theorem \ref{MDPoly}}
Let $t$ be defined by formula (\ref{s2}). For each
$w=(w_1,\dots,w_\ell)\in(\Z^d)^\ell$ we define
\[\mathcal{R}_w=\Bigl\{r\in[1,N_0]\,:\,w_j+\sum_{i=1}^d Q_{ij}(r)e_i \in A \ (1\leq j\leq\ell)\Bigr\},\]
with $N_0=(N/t)^{1/k}$.
Since, for $N$ sufficiently large (depending on the coefficients of $Q_{ij}$),
\[\max_{r\in[1,N_0]}|Q_{ij}(r)|\leq |Q_{ij}(N_0)|\leq 2N\] 
for all $1\leq j\leq \ell$, we will restrict our attention to those $w$ that are contained in the set
\[\mathcal{W}=\{w\in(\Z^d)^\ell\,:\, |\langle w_j,e_i\rangle|\leq 2N-1 \ (1\leq j\leq\ell,\, 1\leq i\leq d)\},\]
and note that $|\mathcal{W}|\leq  (4N)^{\ell d}$ provided that $N$ is sufficiently large.
Since the average
\[\frac{1}{|\mathcal{W}|}\sum_{w\in\mathcal{W}}|\mathcal{R}_w|=\frac{1}{|\mathcal{W}|}\sum_{w\in\mathcal{W}}\sum_{r=1}^{N_0}\prod_{j=1}^\ell 1_{A}\Bigl(w_j+\sum_{i=1}^d Q_{ij}(r)e_i\Bigr)\geq\left(\frac{|A|}{N^d}\right)^\ell\frac{N_0}{4^{\ell d}}\]
for $N$ sufficiently large, it follows that there will necessarily exist $w\in\mathcal{W}$ for which $|\mathcal{R}_w|\geq2$ provided that
\[|A|/N^d\geq (2^{1/\ell}4^dt^{1/\ell k}) N^{-1/\ell k}\]
and $N$ is sufficiently large.


\end{document}